\documentclass[12pt,reqno]{amsart}
\usepackage{amsthm}
\usepackage{amssymb}
\usepackage{latexsym}
\usepackage{multicol}
\usepackage{verbatim,enumerate}
\usepackage[usenames]{color}
\usepackage[colorlinks=true, linkcolor=blue, citecolor=red, urlcolor=blue]{hyperref}
\usepackage{hyperref}
\usepackage{hyperref}
\usepackage{amsmath, amscd}
\usepackage{tikz}
\usetikzlibrary{matrix,arrows,decorations.pathmorphing}
\usetikzlibrary{positioning}

\advance\textwidth by 1.2in \advance\oddsidemargin by -.6in \advance\evensidemargin by -.6in

\numberwithin{equation}{section}
\theoremstyle{plain}

\newtheorem{thm}{Theorem}

\newtheorem{lem}{Lemma}

\theoremstyle{definition}

\newtheorem{defn}{Definition}

\theoremstyle{remark}
\newtheorem{rem}{Remark}

\newcommand{\gen}[1]{\left < #1 \right >}

\newcommand{\ad}{\operatorname{ad}}
\newcommand{\rank}{\operatorname{rank}}

\newcounter{cnt}
 \makeatletter
\def\mydggeometry{\makeatletter\dg@YGRID=1\dg@XGRID=20\unitlength=0.003pt\makeatother}
\makeatother \theoremstyle{remark}

\numberwithin{equation}{section}
\makeatletter
\def\section{\def\@secnumfont{\mdseries}\@startsection{section}{1}%
  \z@{.7\linespacing\@plus\linespacing}{.5\linespacing}%
  {\normalfont\scshape\centering}}
\def\subsection{\def\@secnumfont{\bfseries}\@startsection{subsection}{2}%
  {\parindent}{.5\linespacing\@plus.7\linespacing}{-.5em}%
  {\normalfont\bfseries}}
\makeatother

\begin{document}
\baselineskip 15pt


\title{Image of Lie polynomial of degree $2$ evaluated on Nilpotent Lie algebra}
\author{Niranjan}
\address{ Indian Institute of Science Education and Research, Mohali, Punjab, India}
\email{niranjannehra11@gmail.com, ph17042@iisermohali.ac.in}
\author{Shushma Rani$^{\ast}$}
\address{Indian Institute of Science Education and Research, Mohali, Punjab, India}
\email{shushmarani95@gmail.com, ph16067@iisermohali.ac.in.}
 
\thanks{2020 Mathematics Subject Classification. 17A60,17B05,17B30\\
$^{*}$The corresponding author.}
\keywords{Multilinear Lie polynomials, Nilpotent Lie algebra, Frattini subalgebra, Isoclinism, Breadth}
\maketitle

\begin{abstract} We delineate the image of multilinear Lie polynomial of degree $2$  evaluated on $L$ where $L$ is a finite-dimensional nilpotent Lie algebra over field $k$ with $\dim L' \leq 4$. 

\end{abstract}
\section{Introduction} In recent years, people have had a keen interest in the famous Lvov-Kaplansky Conjecture: \emph{The image of a multilinear polynomial in noncommutative variables over a field $K$ on the matrix algebra $M_n(K)$ is a vector space.} For more details, see \cite{serveyLKC2020},\cite{image2},\cite{imageliemon2017},\cite{imglie}. A variation of the Lvov–Kaplansky conjecture has been formulated in \cite{imglie}: \emph{The image of a multilinear Lie polynomial over a field $K$ on classical Lie algebras is a vector space}. Further, they proved this result for multilinear Lie polynomials of degrees $3$ and $4$ for some classical Lie algebras.

           A non zero multilinear Lie polynomial in two variables is nothing but a non-zero scalar multiple of the Lie bracket of these variables. For multilinear Lie polynomials of degree 2, this problem has been studied by Brown in \cite{brown1963}. He proved that every element of split semi-simple Lie algebra $L$ is a Lie bracket over all infinite fields or finite fields of sufficiently big cardinality. For real semi-simple Lie algebras with some conditions, the same result proved in \cite{dmitri2015}. On simple Lie algebras, this problem can be seen as the Lie theoretical version of a famous Ore conjecture which was proved recently in \cite{ore_2010}. \emph{Ore's conjecture: Every element of a finite simple group is a single commutator.}

On nilpotent Lie algebras, the image set of a non-zero multilinear Lie polynomial of degree 2 is not necessarily a vector space. Smallest such nilpotent Lie algebra over a field $k$ with characteristic not equal to $2$ is $L_{6,21}(0)$, see Section \ref{sec6}. So a natural way to study this question is imposing some constraints on Lie algebra. We imposed a constraint on the spanning set of images, which is nothing but a derived Lie subalgebra $L'=[L, L].$ Let us denote the image set of multilinear Lie polynomial of degree $2$ on $L$ by $w(L).$ For the finite-dimensional nilpotent Lie algebra $L$ over $k$, the breadth of $x \in L$, denoted by $b(x)$, is the rank of the adjoint map $ad x\colon L \rightarrow L$ given by $y \mapsto [x,y].$ The breadth of a Lie algebra $L$, denoted by $b(L)$, is the maximum of the breadths of all elements of $L$ and breadth type is a tuple consists of all possible breadths of elements of $L$. The following theorems are the main results of our paper. 

\begin{thm}\label{dimL'3} Let $L$ be a finite-dimensional nilpotent Lie algebra over field $k$ of characteristic not equal to $2$ such that $\dim L' \leq 3.$ Then $w(L)=L'.$
\end{thm}  

\begin{thm}\label{thmA} Let $L$ be a nilpotent Lie algebra over the finite field $k=\mathbb{F}_q$ of odd characteristic with $\dim L'=4.$ Then $w(L) \neq L'$ if and only if one of the following holds:
	\begin{itemize}
		\item $L$ is $4$-step nilpotent with $\dim (L)=6 $ and $\dim (Z(L))=2.$
		\item $L$ is $3$-step nilpotent with $\dim (L)=7 $ and $\dim (Z(L))=3.$
		\item $L$ is $2$-step nilpotent with $\dim (L)=8 $ along with one of the following:
		\begin{itemize}
			\item Breadth type of $L$ is $(0, 1,3)$ or $(0,1, 2, 3)$.
			\item Breadth type of $L$ is $(0,2,3)$ and does not possesses any generating set $\{u_1,u_2,u_3,u_4\}$ such that $[u_1,u_2]=0=[u_3,u_4].$
		\end{itemize}
	\end{itemize}
	Further, if $w(L) \neq L'$ then each element of $L'$ is sum of at most two elements of $w(L).$
\end{thm}
Image set of commutator word map for finite groups is current topic of research \cite{group2005},\cite{P-group2005},\cite{manoj2021}. Our results are similar to the one given for finite $p$-groups in \cite{P-group2005},\cite{manoj2021}. Most of the arguments presented in Theorem \ref{thmA} are similar to \cite{manoj2021}.

Our paper has been organized in the following way. In Section \ref{sec1}, we will discuss some definitions and notations. In Section \ref{sec2}, we will see some results that we will use in each of the subsequent sections. Section \ref{sec3} will give a proof of Theorem \ref{dimL'3}. Sections \ref{sec4},\ref{sec5} and \ref{sec6} describe our question under investigation for $2$-step, $3$-step, and and $4$-step nilpotent Lie algebras, respectively. In the last section \ref{sec7}, we discuss the proof of our main Theorem \ref{thmA}. 

{\em Acknowledgments. The authors thank Tushar Kanta Naik for helpful discussions. The first author acknowledges the CSIR research grant: 09/947(0084)/2017-EMR-I. The second author acknowledges the CSIR research grant: 09/947(0082)/2017-EMR-I.}

\section{Preliminaries}\label{sec1}
We will assume $L$ is a finite-dimensional Lie algebra over field $k.$
\noindent
\begin{defn} A nonzero {\bf multilinear Lie polynomial} $w$ of degree $2$ is a polynomial over field $k$ that can be written in the form $$w(x,y)=c[x,y]$$
where $0 \neq c \in k.$
\end{defn}

\begin{defn} Let $L$ be a Lie algebra over the field $k$, then the lower central series of $L$ is defined as
$$ L^0 \supset L^1 \supset L^2 \supset \cdots L^i \supset L^{i+1} \supset \cdots \cdots $$ where $L^0=L,\,  L^1=L'=[L,L]$, $L^i=[L, L^{i-1}].$ $L$ is said to be a nilpotent Lie algebra if this lower central series terminates at a finite step, that is, $L^n=0$ for some $n.$ If $n$ is the least non-negative integer with this property, i.e., $L^{n-1} \neq 0$, then $n$ is called the nilpotency class of $L$ or we say, $L$ is {\bf $n$-step nilpotent}. For example, any abelian Lie algebra is $1$-step nilpotent.
\end{defn}

From \cite{tower1973},\cite{marshall1967}, and \cite{ernest1970}, we recall some definitions and results that are required in our proofs.

\begin{defn}  An element $x \in L$ is called a non-generator of $L$ if, whenever $S$ is a subset of $L$ such that $S$ and $x$ together generate $L$, then $S$ alone generates $L.$
\end{defn}

\begin{defn}
 The {\bf Frattini subalgebra}, $F(L)$ of a Lie algebra, is defined as the intersection of the maximal subalgebras of $L.$
\end{defn}

\begin{rem} Any maximal subalgebra $M$ of a nilpotent Lie algebra $L$ is an ideal of $L.$ The Frattini subalgebra $F(L)$ of a nilpotent Lie algebra is equal to the derived subalgebra of $L$ and $F(L)$ is the set of non-generators of $L$. Further, the dimension of $L/L'$ is the minimal number of generators of $L.$
\end{rem}
 

 \begin{defn}
 	Two Lie algebras $L_1$ and $L_2$ over field $k$ are said to be isoclinic, whenever there exist isomorphisms $\eta: L_1/Z(L_1) \longrightarrow L_2/Z(L_2)$ and $\tau: L_1' \longrightarrow L_2'$ such that if $$\eta(u_1+Z(L_1))= v_1+Z(L_2), \,\, 
 	\eta(u_2+Z(L_1))= v_2+Z(L_2)$$ then $\tau([u_1,u_2])=[v_1,v_2]$. The pair $(\eta,\tau)$ is called isoclinism between $L_1$ and $L_2$.
 
 \end{defn}

\begin{defn} A Lie algebra $L$ is said to be {\bf stem (or pure)} if $Z(L) \subseteq L'.$
\end{defn}

\begin{rem}
Each isoclinism family of finite-dimensional Lie algebra contains a stem Lie algebra of minimal dimension.
\end{rem}
\begin{defn}
A Lie algebra $L$ is said to be {\bf the central product} of $M$ and $N$, if $L=M+N$, where $M$ and $N$ are ideals of $L$ such that $[M, N]=0$ and $M \cap N \subseteq Z(L).$ 
\end{defn}

\begin{defn}
For the finite-dimensional nilpotent Lie algebra $L$ over $k$, the breadth of $x \in L$, denoted by $b(x)$, is the co-dimension of $C_L(x)$, the centralizer of $x$ in $L$, can also be thought of as the rank of the adjoint map $ad x\colon L \rightarrow L$ given by $y \mapsto [x,y].$ The maximum of the breadths of all elements of $L$ is called {\bf breadth of a Lie algebra} $L$, denoted by $b(L)$. For a Lie algebra $L$, \textbf{breadth type} is a tuple consists of all possible breadths of elements of $L$.
\end{defn}

\section{Results}\label{sec2}
The proof of the following lemmas are straightforward.
\begin{lem}
Let $L_1$ and $L_2$ be two isoclinic finite-dimensional Lie algebras. Then $L_1$ and $L_2$ are of the same breadth type.
\end{lem}


\begin{lem}\label{prelem}
Let $L_1$ and $L_2$ be two isoclinic finite-dimensional nilpotent Lie algebras over field $k.$ Then $w(L_1)=L_1'$ if and only if $w(L_2)=L_2'.$ 
\end{lem}

In the light of above, for the classification of Lie algebras of a given breadth type, it is enough to consider stem Lie algebras of that breadth type. From now onwards, \textbf{we will consider only stem Lie algebras}.

\begin{lem}\label{maximal_nilpotency_class}
For a finite-dimensional Lie algebra $L$ over field $k$ which is atleast $4$- step nilpotent, $Z(L)\cap L'$ cannot be maximal in $L'.$
\end{lem}
\begin{proof}
Suppose, if possible that $Z(L)\cap L'$ is maximal in $L'.$ Then $[L':Z(L)\cap L']=1$. So, $L/ (Z(L)\cap L')$ is $2$-step nilpotent, which implies that $L^2 \subseteq Z(L)$, a contradiction to given hypothesis.
\end{proof}

The proofs of the following three lemmas are straightforward.

\begin{lem}\label{prelem3}
Let $L$ be a finite-dimensional nilpotent Lie algebra over field $k=\mathbb{F}_q$ and $I \subseteq Z(L)\cap L'.$ If there exists $u_1, u_2, \cdots, u_n \in L$ such that $L'/I =\bigcup\limits_{i=1}^{n}[u_i+I, L/I]$ and $I \subseteq \bigcap\limits_{i=1}^{n}[u_i,L]$, then $L'= \bigcup\limits_{i=1}^{n}[u_i,L].$
\end{lem}



\begin{lem}\label{prelem4}
Let $L$, a finite-dimensional Lie algebra, be a direct sum of Lie algebras $L_1$, $L_2.$ Then $w(L)=L'$ if and only if $w(L_1)=L_1'$ and $w(L_2)=L_2'.$ 
\end{lem}

\begin{lem}\label{prelem5} Let $L$ be a finite-dimensional nilpotent Lie algebra over field $k$ and $I$ be an ideal of $L$ of dimension $1$ contained in $Z(L)$ such that $(L/I)'=w(L/I).$ Then each element of $L'$ is sum of at most two elements of $w(L).$
\end{lem}


The following three results from \cite{misra2015} and \cite{ssk2021}, classify the breadths of Lie algebras.

\begin{thm}\label{breadth1_classify} From \cite[Theorem 2.3]{misra2015}, $b(L)=1$ if and only if $\dim L'=1.$
\end{thm}

\begin{thm}\label{breadth2_classify} From \cite[Theorem 3.1]{misra2015}, let $L$ be a finite-dimensional nilpotent Lie algebra over field $k$ of characteristic not equal to $2.$ Then $b(L)=2$ if and only if one of the following conditions holds:
\begin{itemize}
\item[(1)]$\dim L'=2$, or
\item[(2)]$\dim L'=3$ and $\dim (L/Z(L))=3.$
\end{itemize}
\end{thm}

\begin{thm}\label{breadth3_classify} From \cite[Theorem 3.1]{ssk2021}, let $L$ be a finite-dimensional nilpotent Lie algebra over $k=\mathbb{F}_q.$ Then $b(L)=3$ if and only if one of the following holds:
\begin{itemize}
\item[(i)] $\dim (L')=3$ and $[L:Z(L)] \geq 4.$
\item[(ii)] $\dim (L')\geq 4$ and $[L:Z(L)] = 4.$
\item[(iii)] $\dim (L')=4$ and there exists an ideal $I \subseteq Z(L)$ of $L$ with $\dim I=1$ and $[L/I:Z(L/I)] = 3.$
\end{itemize}
\end{thm}

\begin{rem}\label{remark} Let $L$ be a finite-dimensional nilpotent Lie algebra over $\mathbb{F}_q$ with $\dim(L')=4.$ Then by the above theorems it follows that $b(L) \geq 3.$ We will use this information throughout the following without reference.
\end{rem}

The following theorem describes that every finite-dimensional nilpotent Lie algebra with $1$-dimensional derived Lie subalgebra, is Heisenberg Lie algebra up to isoclinism. 

\begin{thm}\label{heisenberg_thm} Let $L$ be an $n$-dimensional nilpotent Lie algebra over field $k$ with $b(L)=1.$ Then $L$ has a basis $\{x_1,y_1,x_2,y_2,\cdots,x_m,y_m,z_1, \cdots, z_{n-2m} \}$ with $[x_i,y_j]=\delta_{ij}z_1$ and $[z_j,L]=0.$ 
\end{thm}

\begin{lem}\label{ext_lemma} Let $L$ be a finite-dimensional nilpotent Lie algebra over field $k$ and $I$ be a central ideal of $L$ such that $[L/I:Z(L/I)]=3$, then

{\it{(a)}} If  $\frac{L/I}{Z(L/I)}$ is a $3$-generator Lie algebra, then we can choose except three generators $x,y,z$ (say) of $L$, all other generators $u_1, u_2, \cdots, u_n, n \geq 0$, so that $[u_i, L] \subseteq I.$ If $\dim I=1$, then $[u_i, L]=I.$ Furthermore, $L/I$ is $2$-step nilpotent.

{\it{(b)}} If  $\frac{L/I}{Z(L/I)}$ is a $2$-generator Lie algebra, then  we can choose, except two generators $x,y$ (say) of $L$, all other generators $u_1, u_2, \cdots, u_n, n \geq 0$, in such a way that $[u_i, L] \subseteq I.$ If $\dim I=1$, then $[u_i, L]=I.$
\end{lem}
\begin{proof}
Since  $\dim \frac{L/I}{Z(L/I)}=3$, the minimal number of generators of $\frac{L/I}{Z(L/I)}$ can be $2$ or $3.$
 {\it{(a)}} If $\frac{L/I}{Z(L/I)}$ is a $3$-generator Lie algebra, say $\bar{x},\bar{y},\bar{z}$ be $3$-generators that do not belong to $Z(L/I)$, and $\bar{u_1}, \bar{u_2}, \cdots, \bar{u_n}, n \geq 0$ be other generators of $(L/I)$ that belong to $Z(L/I).$ we can choose unique representatives of pre-images of each of these generators in natural map $L \longrightarrow L/I.$ Since $I \subset Z(L) \subset L'$, we have $x,y,z, u_1, u_2, \cdots, u_n, n \geq 0$ as generators of $L$ such that $[u_i, L] \subseteq I.$ Observe that $u_i \notin Z(L)\subset L'$ as $u_i$'s are generators of $L.$ $0 \neq [u_i,L]\subseteq I$ and $\dim I=1$ implies $[u_i,L]= I.$ Since $\frac{L/I}{Z(L/I)}$ is a $3$-generator Lie algebra of dimension $3$, so $(\frac{L/I}{Z(L/I)})'=0$, i.e., $L'/I \subseteq Z(L/I)$ which implies that $L/I$ is $2$-step nilpotent.
 
 {\it{(b)}} If $\frac{L/I}{Z(L/I)}$ is $2$-generator Lie algebra, then using similar arguments as above, the proof follows.
\end{proof}
The following theorem takes the edge off our study to Lie algebras of small dimensions.
\begin{thm}\label{keyresult} Let $L$ be a finite-dimensional nilpotent Lie algebra over field $k= \mathbb{F}_q$ of odd characteristic, with $Z(L) \subseteq L'$, $\dim L'= 4$ and $b(L)=3.$ If $L$ is $3$-step nilpotent, then one of the following holds:
\begin{enumerate}
\item[(i)] There exists a $6$-dimensional ideal $M$ of $L$, generated by $2$-elements such that $M'=L'$ and nilpotency class of $M$ is same as that of $L$. If $\dim L \geq 7$, then $L=M+N$ where $N$ is the Lie subalgebra of $L$ with $ \dim N' \leq 1.$ Further, if $N$ is non-abelian, it is isoclinic to the Heisenberg Lie algebra over $k.$ 
\item[(ii)] There exists a $5$-dimensional ideal $M$ of $L$, generated by $2$-elements such that $M'\subset L'$ and nilpotency class of $M$ is same as that of $L$. If $\dim L \geq 7$, then $L$ is the central product of $M$ and a $2$-step nilpotent Lie subalgebra $N$ that is isoclinic to the Heisenberg Lie algebra over $k.$ 
\item[(iii)] There exists a $7$-dimensional ideal $M$ of $L$, generated by $3$-elements such that $M'= L'$ and nilpotency class of $M$ is same as that of $L$. If $\dim L \geq 8$, then $L=M+N$ where $N$ is the Lie subalgebra with $ \dim N' \leq 1.$ Further, if $N$ is non-abelian, it is isoclinic to the Heisenberg Lie algebra over $k.$ 
\end{enumerate} 
If $L$ is $4$-step nilpotent, then only $(i)$ holds. 
\end{thm}
\begin{proof}
Since $b(L)=3$ and $\dim L'=4$, by Theorem \ref{breadth3_classify}, either $[L: Z(L)]=4$ or there exists a one-dimensional ideal $I$ of $L$ such that $[L/I:Z(L/I)] = 3.$
\noindent
If $[L: Z(L)]=4$ then minimal number of generators of $L$ can be $2,3$ or $4.$ If $L$ is $4$-generator Lie algebra then $L'=z(L)$, i.e.,  $L$ is $2$-step nilpotent, which is contradiction to given hypothesis.

  If $L$ is a $2$-generator Lie algebra, then $L$ itself is $6$-dimensional satisfying (i). If $L$ is a $3$-generator Lie algebra, then $L$ itself is $7$-dimensional satisfying (iii). Further, $[L':Z(L)]=1$, so $L$ is at most $3$-step nilpotent by Lemma \ref{maximal_nilpotency_class}. But it cannot be $2$-step as otherwise $L' \subseteq Z(L).$ Hence, $L$ is $3$-step nilpotent.

Now, we will consider the other case when there exists $1$-dimensional ideal $I$ of $L$ with $[L/I:Z(L/I)] = 3.$ Then, the possible minimal number of generators of $\frac{L/I}{Z(L/I)}$ can be $2,3.$ We will discuss these cases here. First, assume that $L$ is $3$-step nilpotent. Then one of the following holds:

 (a) $ I= L^2$ \hspace{2cm}
(b) $ I\subsetneq L^2$ 
\hspace{2cm} (c) $ I \nsubseteq L^2.$

{\textbf{\it (a)}} $I= L^2$, i.e., $L/I$ is $2$-step nilpotent, which implies $\frac{L/I}{Z(L/I)}$ is 3-generator Lie algebra. Indeed, if $\frac{L/I}{Z(L/I)}$ is 2-generator Lie algebra, then $L'/I=L'/L^2$ is $1$-generator Lie algebra, so $\dim L'/L^2=1$ and $\dim L^2= \dim I=1$ implies $\dim L'=2$, which is contradiction to given hypothesis. By Lemma \ref{ext_lemma}, it follows that, except for three generators $x,y,z$ of $L$, all other generators $u_1,u_2, \cdots, u_n, n \geq 0$, are such that $[u_i,L] = I$ for all $1 \leq i \leq n.$
 
Set $M:=<x,y,z>$ and $N:=\left\langle u_1,u_2,\cdots, u_n \right\rangle$, which are Lie subalgebras of $L.$ Observe that $\dim(M'+I/I)=3$ as $L'/I=M'+I/I$. We claim that 
$I \subseteq M'.$ $C_L(u_i)$ is maximal in $L$ as $[u_i,L]= I$ for all $1 \leq i \leq n$. So $L' \subseteq C_L(u_i)$ as $L'=F(L)$ is the intersection of all maximal subalgebras of $L$.
 
  Since $L^2=[L',L]=I \subseteq Z(L)$, for any generator $v \in I$ can be written as 
$$v=[\alpha_1w_1+\alpha_2w_2+\alpha_3w_3, \beta_1 x+\beta_2 y+\beta_3 z] \in L^2,$$ where $L'+I/I=\gen{w_1+I, w_2+I, w_3+I}$. Thus, $v \in [M',M] = M^2 \subseteq M'$, i.e., $I \subseteq M'.$ Hence, our claim is proved. Therefore $M'/I=L'/I$ i.e. $M'=L'.$ Thus, $M$ is $7$-dimensional, $3$-step nilpotent as $0\neq v \in M^2 \implies M^2 \neq 0.$ Since $N=\langle u_1, \cdots, u_n\rangle$ and $[u_i,L]= I$ for all $1 \leq i \leq n$, implies that $N'\subseteq  I$, i.e.,  $N$ is at most $2$-step nilpotent. Now, $L' =M' \subseteq M$ and $N$ acts on $M$ by ad-map. $M$ is ideal of $L$, $L=N+M$ and $[N,M]\subseteq I \subset L.$ If $N$ is non-abelian, then $N$ is isoclinic to the Heisenberg Lie algebra by Theorem \ref{heisenberg_thm}.

{\textbf{\it (b)}} If $I\subset L^2$ then $L/I$ is $3$-step nilpotent and $[L/I: Z(L/I)] = 3.$ So $\frac{L/I}{Z(L/I)}$ is a $2$-generator Lie algebra by Lemma \ref{ext_lemma}. We can conclude that L can be generated by $\{x, y, u_1,\cdots, u_n\}$ so that $ [u_i, L] =I.$ Now, using the same arguments as in the preceding case, required result follows by taking $M := \langle x,y \rangle$ and $ N:=\langle u_1,...,u_n \rangle $, where $ \dim M =6.$

{\textbf{\it (c)}} If $I \nsubseteq L^2$ then $I \cap L^2 = 0 $ as $\dim I=1.$ Thus, $L/I$ is also $3$-step nilpotent. By using the same arguments as in the above case, $L/I$ is a $2$-generator Lie algebra and $L=\gen{x, y, u_1, ..., u_n}$ such that $ [u_i, L] =I.$ Let $M_1 := \langle x,y \rangle.$ Observe that $L'/I=M_1'+I/I$  has dimension 3,  $\dim (M_1 +I/I) =5$ and $M_1$ and $L$ have same nilpotency class. Since $M_1$ is a 2-generator, $M_1'/ M_1^2$ is 1-dimensional. Thus, $M_1$ cannot contain $I$, which implies $\dim(M_1) =5.$ If $M_1 \subseteq C_L(u_i)$, for all $1 \leq i \leq n$, i.e., $[x,u_i] = 0 = [y,u_i],\, \forall \, 1 \leq i \leq n$ then $N_1:=\langle u_1,...,u_n \rangle $ with $N_1' =I$ is isoclinic to the Heisenberg Lie algebra over $k$, and hence L is the central product of $M_1$ and $N_1.$ Therefore, $M=M_1$ and $N=N_1$ are the required Lie subalgebras. If $M_1 \nsubseteq C_L(u_i)$ for some $i$, i.e., assume $[u_i, M_1]\neq 0$ for some $i.$ Thus, $0 \neq [u_i, M_1]\subset [u_i, L]=I$ implies that $[u_i, M_1] = I$, and the Lie subalgebra $M :=\langle x,y,u_i \rangle $ of L is 7-dimensional.
We can easily see that $M $ and $N := \langle u_1, \cdots ,u_i-1,u_i+1,\cdots,u_n \rangle $ are the required Lie subalgebras of $L.$

Now, assume that $L$ is $4$-step nilpotent. Then either $I = L^3$ or $I \neq L^3.$ We claim that there is no $4$-step nilpotent Lie algebra $L$ such that $I \neq L^3.$ If such an $L$ exists, then $L/I$ is $4$-step nilpotent, which is not possible, as $(L/I)/ Z(L/I)$, being of dimension $3$, can be at most $3$-step nilpotent. Therefore, we can assume that $I = L^3$ which implies that $L/I$ is $3$-step nilpotent. Therefore, L can be generated by $\{x, y, u_1, ..., u_n\}$ such that $[u_i, L] =I$ by \ref{ext_lemma}. The required result follows by assuming $M :=\langle x,y \rangle $ and $N:=\langle u_1,...,u_n \rangle$, where $ \dim M= 6.$ This completes the proof of the theorem.
\section{Proof of Theorem \ref{dimL'3}}\label{sec3}

\begin{proof} If $\dim(L')$ is $1$ or $2$, then result holds trivially by Theorem \ref{breadth1_classify} and $b(L)\leq  \dim L'$. Now, assume that $\dim L'=3$.
	Since $Z(L) \subseteq L'$, thus $b(L)=2$ or $3.$ Thus, the following two cases arises:
	
	{\it{ Case(1)}} If $b(L)=3$ then there exists an element $x \in L$ such that $\rank(\ad x)=3$, i.e., $L'=[x,L]$. Thus $L'=w(L).$
	
	{\it{ Case(2)}} If $b(L)=2$ then $\dim (L/Z(L))=3$ by Theorem \ref{breadth2_classify}. Also $[L':Z(L)] \leq [L:Z(L)]=3$. But $[L': Z(L)] \neq 2$ or $3$ otherwise $L'=Z(L)$ or $L'=L$ respectively, which is not possible. So we have the following subcases. \medskip
	
	{\it{subcase(2a)}} If $L'=Z(L)$ then $[L:L']=3$. Let $L=\gen {x,y,z}$ and $L'=\gen {[x,y],[x,z],[y,z]}$ as $L^2=0$, i.e., $L$ is $2$-step nilpotent. 
	$$\begin{aligned}
	&\text{If } \alpha \neq 0 \text{ then } \alpha[x,y]+ \beta [x,z]+\gamma [y,z]= [x-\frac{\gamma}{\alpha}z, \alpha y+ \beta z].\\
	&\text{If } \alpha= 0 \text{ then }
	\beta [x,z]+\gamma [y,z]= [z, -\beta x-\gamma y].
	\end{aligned}$$
	Thus, $w(L)=L'.$
	
	{ \it{subcase(2b)}} If $[L': Z(L)]=1$ then $L$ is at most $3$-step nilpotent by Lemma \ref{maximal_nilpotency_class}. But $L$ cannot be $1 $ or $2$- step nilpotent as $L'=0$, or $L'=Z(L)$ respectively, which contradicts the hypotheses. So $L$ is $3$-step nilpotent. But $[L:L']=2$ implies $L=\gen {x,y}.$ Thus, $L'=\gen {[x,y],L^2}$, so $\dim (L'/L^2)=1$ and $\dim L^2=2$, i.e., $L'=\gen {[x,y],[x,[x,y]],[y,[x,y]]}.$ In the above calculation, replacing $z$ with $[x,y]$ gives $w(L)=L'.$
	\end{proof}

\section{$2$-step nilpotent Lie algebras}\label{sec4}
Observe that if $L$ is a $2$-step nilpotent Lie algebra with $\dim L' =4$ then $Z(L)=L'$ and $L$ is minimally generated by $4$ elements i.e., $\dim L \geq 8.$ The following lemma investigates our question about an $8$-dimensional nilpotent Lie algebra. 

\begin{lem} \label{2-step_dim8} 
	Let $L$ be an $8$-dimensional, $2$-step nilpotent Lie algebra over $k=\mathbb{F}_q$, field of odd characteristic such that $\dim L'=4.$
	
	{\it{(1)}} If breadth type of $L$ is $(0, 1,3)$ or $(0,1, 2, 3)$, then $w(L) \neq L'.$
	
	{\it{(2)}} Let $L$ is of breadth type $(0, 2, 3).$ Then $w(L) = L'$ if and only if $L$ admits a generating set $\{u_1, u_2, u_3, u_4\}$ such that $[u_1, u_2] = 0 = [u_3, u_4].$
	
	{\it{(3)}} If breadth type of $L$ is $(0, 3)$ then $w(L) = L'.$ 
	
	 Further, if $w(L) \neq L'$, then each element of $L'$ is sum of at most two elements of $w(L).$
	\end{lem}
\begin{proof}
	{\it{(1)}} Since $L$ has breadth $1$-element, say $w$, we can always extend $\{w\}$ to a generating set $\{x, y, z, w\}$ for $L$ such that $L'=\langle [x,y], [x,z], [y,z], [z,w]\rangle.$ We claim that $[x,y]+[z,w] \notin w(L).$ Suppose, if possible, that $[x,y]+[z,w] \in w(L)$, i.e., 
	$$[x,y]+[z,w] = [{\alpha_1}x +{\alpha_2}y+ {\alpha_3}z+{\alpha_4 }w , {\beta_1}x+{\beta_2}y+{\beta_3}z+{\beta_4}w],$$
	where $\alpha_i, \beta_j \in k$ for $1 \le i, j \le 4.$ On comparing both sides, we get,
	\begin{eqnarray}
		\beta_2 \alpha_1 - \alpha_2 \beta_1 & = & 1, \label{eqn1} \\     
		\beta_3 \alpha_1 - \alpha_3 \beta_1 &=& 0, \label{eqn3}\\
		\beta_3 \alpha_2 - \alpha_3 \beta_2 &=& 0, \label{eqn2}\\
		\beta_4 \alpha_3 - \alpha_4 \beta_3 &=&1.\label{eqn4}
	\end{eqnarray}
	
	If $\beta_3 \neq 0$ then $\alpha_2 = {\alpha_3 \beta_2}{\beta_3}^{-1} $ and $\alpha_1 = {\alpha_3\beta_1}{\beta_3}^{-1}$ by Equations \eqref{eqn2} and \eqref{eqn3} respectively, which contradicts Equation \eqref{eqn1}. So, $\beta_3=0.$ From \eqref{eqn4}, we get $\alpha_3 \neq 0.$ Then from \eqref{eqn2} and \eqref{eqn3}, we get $\beta_1=\beta_2=0$, which don't satisfy \eqref{eqn1}. Therefore, there is no solution of the above system. Hence, $[x,y]+[z,w] \notin w(L).$ 
\medskip

	{\it{(2)}} Let $L$ has a generating set $\gen{u_1,u_2,u_3,u_4}$ such that $[u_1,u_2]=0$ and $[u_3,u_4]=0$. So, $L'=\gen{[u_1,u_3],[u_2,u_3],[u_1,u_4],[u_2,u_4]}$. Consider
	$$\alpha[u_1,u_3]+\beta[u_2,u_3]+\gamma[u_1,u_4]+\delta[u_2,u_4]=[\gamma u_1+\delta u_2+u_3,-\alpha u_1-\beta u_2+u_4]$$ where $\alpha,\beta,\gamma,\delta \in k$. Hence $w(L)=L'$. To prove the converse side, suppose if possible that $L$ does not posses any generating set $\gen{u_1,u_2,u_3,u_4}$ such that $[u_1,u_2]=0=[u_3,u_4]$. But breadth type of $L$ is $(0,2,3)$, i.e., there exists an $x \in L$ such that $b(x)=2$ and there exist other generators $y,z,w \in L$ such that $[x,y] \neq 0, \, [x,w] \neq 0$ and $[x,z]=0$. By our assumption hypothesis, $[y,w]\neq 0$. Since $L$ has an element of breadth $3$, so we can assume that $b(w)=3$, i.e., $[z,w] \neq 0$. Also $[y,z] \neq 0$ otherwise $b(z)=1$ which is not possible for any element of $L$. Therefore, we can always choose a generating set $\gen{x,y,z,w}$ of $L$ such that $[x,z]=0$, i.e., only one commutator of generating elements is trivial. Observe that if $[z,w]$ cannot be written as a linear combination of remaining basic commutators then $[x,y]+[z,w] \notin w(L)$  by doing similar calculations as in the above case(1). So, we can assume that
	$$[z, w]= {\lambda_1}[x, y]+ {\lambda_2}[y, z]+ {\lambda_3}[y, w]+ {\lambda_4}[x,w], \text{ where } \lambda_1,\lambda_2,\lambda_3,\lambda_4 \in k$$
	\begin{equation}\label{cls2eqn1}
		\implies [y,{-\lambda_1}x+{\lambda_2}z+{\lambda_3}w]+[w, {-\lambda_4}x+z]= 0.
	\end{equation}
	If $\lambda_3 \neq 0$ then above equation can be written as $[y, {-\lambda_1}x + {\lambda_2}z + {\lambda_3}w] + {\lambda_3}^{-1} [{-\lambda_1}x + {\lambda_2}z + \lambda_3w, {-\lambda_4}x+z]= 0.$ Taking $w'=	{-\lambda_1}x + {\lambda_2}z + {\lambda_3}w $, we get $[w', -y- {\lambda_4 \lambda_3^{-1}} x + {\lambda_3}^{-1}z]=0.$ Take $y'=-y-{\lambda_4\lambda_3^{-1}}x+{\lambda_3}^{-1}z$, we get a generating set $\{ x,y',z,w'\}$ for $L$ such that $L'=\langle [x, y'],[y', z],[y', w'],[z,w'] \rangle$ and $[x, z]=0=[y',w']$, which contradicts our assumption hypothesis. So, $\lambda_3=0.$ Therefore, \eqref{cls2eqn1} reduces to $[y, {-\lambda_1}x+{\lambda_2}z ]+[w, z-{\lambda_4}x]= 0.$ Now, replace $z$ by $z' := z-{\lambda_4}x .$ An easy calculation gives 
	$$[y,{-\lambda_1}x+{\lambda_2}z]=[y,{-\lambda_1}x+{\lambda_2}(z'+{\lambda_4}x)]= [y,(\lambda_4\lambda_2 - \lambda_1)x+{\lambda_2}z'].$$ 
	Equation \eqref{cls2eqn1} gives,
	\begin{equation}\label{cls2eqn1a}
		[y,(\lambda_4 \lambda_2 -\lambda_1) x + {\lambda_2}z'] + [w,z']=0.
	\end{equation}
	
	We claim that $\lambda_4 \lambda_2 - \lambda_1 \neq 0.$ Suppose, $\lambda_4 \lambda_2 - \lambda_1=0.$ Then \eqref{cls2eqn1a} reduces to $[y,{\lambda_2} z'] + [w, z']=0,$ i.e., $[{\lambda_2}y+w, z']=0.$ Taking $w' := {\lambda_2}y+w$, we get a generating set $\{x, y, z', w' \}$ of $L$ such that $[x,z']=0=[w',z']$, which implies $b(z')=1$, contradiction to the given hypothesis.
	
	So, we now assume $\lambda_4 \lambda_2- \lambda_1 \neq 0.$ Taking $x' := (\lambda_4 \lambda_2 - \lambda_1)x+{\lambda_2}z'$, we get a new generating set $\{ x', y, z', w \}$ such that $L'= \langle[x', y],[y, z'],[y, w],[x',w] \rangle$, $[x',z']=0$ and, by \eqref{cls2eqn1a}, $-[x', y]=[z', w].$  Now, we claim that ${\mu_1}[y,z']+{\mu_2}[x', w] \notin w(L)$ for some $\mu_1,\mu_2 \in k^*.$ Suppose if possible that ${\mu_1 }[y, z']+{\mu_2}[x',w] \in w(L)$ for all $\mu_1, \mu_2 \in k^*.$ Thus,
	$${\mu_1}[y,z']+{\mu_2}[x',w]=[{\alpha_1}x'+{\alpha_2}y+{\alpha_3}z'+{\alpha_4 }w, {\beta_1}x' + {\beta_2}y + {\beta_3}z' + {\beta_4}w],$$
	where $\alpha_i,\beta_j \in k$ for $1 \le i,j \le 4.$ Expanding the Lie bracket on the right hand side and equating terms on both sides, we get
	\begin{eqnarray}
		\alpha_1 \beta_4 - \alpha_4 \beta_1 &=& \mu_2,\label{eqn16}\\
		\alpha_2 \beta_4  - \alpha_4 \beta_2 &=& 0, \label{eqn15}\\
		\alpha_2 \beta_3  - \alpha_3 \beta_2   &=& \mu_1, \label{eqn14}\\
		\alpha_1 \beta_2  - \alpha_2 \beta_1 - \alpha_3 \beta_4  + \alpha_4
		\beta_3 & = & 0. \label{eqn13} 
	\end{eqnarray}
	
	If $\alpha_2 = 0$ then $\alpha_3
	\beta_2=-\mu_1$, $\alpha_4 = 0$ and $\alpha_1 \beta_4 = \mu_2$ by Equations \eqref{eqn14},  \eqref{eqn15} and \eqref{eqn16} respectively. Since $\alpha_1 \neq 0 \neq \alpha_3$, therefore $\beta_2 = -\mu_1 \alpha_3^{-1}$ and $\beta_4 = \mu_2 \alpha_1^{-1}$ gives $\mu_1 \alpha_1 \alpha_3^{-1}+\mu_2 \alpha_3 \alpha_1^{-1}=0$ by Equation \eqref{eqn13}. Thus, $(\alpha_1 \alpha_3^{-1})^2 = - \mu_1^{-1} \mu_2$, which is not true as we can always choose $\mu_1,\mu_2 \in k^*$ such that $ -\mu_1^{-1}
	\mu_2$ is non square. So, we can assume that $\alpha_2 \neq 0.$ If $\alpha_4 = 0$, then \eqref{eqn15} implies that $\beta_4 =0$, which contradicts \eqref{eqn16}. So, finally assume that both $\alpha_2$ and $\alpha_4$ are non zero. By solving above equations, we get, $\beta_4=\alpha_2^{-1} \alpha_4 \beta_2$, $\beta_1=\alpha_2^{-1}\alpha_1\beta_2-\alpha_4^{-1}\mu_2$, $\beta_3=\alpha_2^{-1}\alpha_3 \beta_2-\alpha_4^{-2}\alpha_2\mu_2$. Putting these values of $\beta_i$'s in \eqref{eqn14} gives $-\mu_1\mu_2^{-1}=(\alpha_2\alpha_4^{-1})^{2}$, which is not true, as we can always choose $\mu_1, \mu_2 \in k^*$ such that $-\mu_1 \mu_2^{-1}$ is non square.\medskip 
	
	{\it{(3)}} Since $L$ has breadth type $(0,3)$, its presentation given in \cite[Theorem 6.4]{riju2021}, is 
	\[ L = \langle x, y, z, w \mid  [z, w]= -[x, y], [x, z] =-r [y,  w]\rangle \] where $r$ is any non-square in $k^*$. For any given $\lambda_i \in k, \, 1 \leq i \leq 4$, we have to find existence of $\alpha_i, \beta_i \in k$ such that 	$${\lambda_1}[x, y] +{\lambda_2}[y, z]+ {\lambda_3}[y, w]+ {\lambda_4}[x, w] = [{\alpha_1}x+{\alpha_2}y+{\alpha_3}z , \  {\beta_1}x+ {\beta_2}y+{\beta_3}z+{\beta_4}w].$$ Opening the Lie bracket on the right hand site and comparing the terms, we get 
	\begin{eqnarray}
		\alpha_1 \beta_2 - \alpha_2 \beta_1 - \alpha_3 \beta_4   &=& \lambda_1,
		\label{eqn17}\\
		\alpha_2 \beta_3  - \alpha_3 \beta_2 &=& \lambda_2 , \label{eqn18}\\
		\alpha_2 \beta_4 -r (\alpha_1 \beta_3  - \alpha_3 \beta_1)   &=& \lambda_3,
		\label{eqn19}\\
		\alpha_1 \beta_4   & = &\lambda_4. \label{eqn20}
	\end{eqnarray}
	Taking $\beta_i$'s as variable, we will show that above system of equations has a solution. By an easy calculation, we get, $\beta_4=\alpha_1^{-1}\lambda_4, \, \beta_3=\alpha_2^{-1}(\alpha_3\beta_2+\lambda_2), \, \beta_1=(r\alpha_3)^{-1}\lambda_3 -(r\alpha_1\alpha_3)^{-1}\alpha_2\xi+(\alpha_2\alpha_3)^{-1}\alpha_1\lambda_2+\alpha_2^{-1}\alpha_1\beta_2$, putting these values in \eqref{eqn17}, we get
	$$r \lambda_2 \alpha_1^2 + (\lambda_3 \alpha_2 + r \lambda_1 \alpha_3) \alpha_1 - \lambda_4(\alpha_2^2 - r \alpha_3^2) = 0$$ which is quadratic equation in $\alpha_1$. If discriminant of this quadratic equation is either zero or quadratic residue then above system of equations has a solution in the field $k$. But discriminant is \begin{equation}\label{eqn21}
		(\lambda_3^2 + 4r\lambda_2 \lambda_4) \alpha_2^2 + 2r\lambda_1 \lambda_3 \alpha_2 \alpha_3 +
		r^2(\lambda_1^2 - 4 \lambda_2 \lambda_4 ) \alpha_3^2.
	\end{equation}  is of the form $f_1 \alpha_2^2 + f_2 \alpha_2 \alpha_3 + f_3 \alpha_3^2$, where $f_1,f_2,f_3 \in k.$ But we can find $\alpha_2, \alpha_3 \in k$ such that \eqref{eqn21} is either zero or a quadratic residue. Hence proved. \medskip

Further, let $I$ be any $1$-dimensional Lie subalgebra of $L'$. By using Theorem \ref{dimL'3}, $(L/I)'=w(L/I)$. By Lemma \ref{prelem5}, the required result follows.	
\end{proof}

\begin{lem} \label{cl2lem2} 
Let $L$ be a  finite-dimensional $2$-step nilpotent Lie algebra of dimension at least  $9$ over $k=F_q$, field of odd characteristic, such that $\dim L'=4$ and $Z(L) =L'.$ Then $w(L)=L'.$ 
\end{lem}
\begin{proof}
	Since $b(L) \leq \dim L'$, i.e., $b(L) \leq 4$. If $b(L)=4$ then there is nothing to prove. So assume $b(L)=3$. By Theorem \ref{breadth3_classify}, there exists one dimensional ideal $I$ of $L$ such that $[L/I:Z(L/I)]=3$. Using similar arguments of Lemma \ref{ext_lemma}, we can assume $$L=\gen{x,y,z,u_1,u_2,\ldots, u_n}$$ where $[u_i,L]=I$ for $1 \leq i \leq n$ and $n \geq 2$ as $[L:L'] \geq 5$. Set $M=\gen{x,y,z}$. Observe that $M$ is $6$-dimensional Lie subalgebra of $L$ with $\dim L'=3$. Therefore, $w(M)=M'$ by using Theorem \ref{dimL'3}.
	
	If $[u_i,M]=0$ for all $1 \leq i \leq n$ then $L$ can be written as a central direct product of $M$ and $n$-generator Lie algebra isoclinic to an Heisenberg Lie algebra generated by $\{u_1, \ldots, u_n\}$ as $Z(L)=L'$. Hence $w(L)=L'$ by using Lemma \ref{prelem4}.
	
	If $[u_i,M]=I$ for some $i \in \{1,\ldots,n\}$. Assume $u_i=u_1$ by re-indexing the set $\{u_j: 1\leq j \leq n\}$. For notational convenience, set $w:=u_1$. Since $C_L(w)$ is a maximal subalgebra of $L$, we can suitably modify te generators $x,y,z$ such that $L'=\gen{[x,y],[x,z],[y,z],[z,w]}$ with $[x,w]=0=[y,w]$ and $I=\gen{[z,w]}$. By suitable modification of $u_i's$, we can assume that $[z,u_i]=0$ for all $2 \leq i \leq n$. For $\alpha,\beta, \gamma,\delta \in k$,
	\[
	\begin{aligned}
		&\text{ If } \delta \neq 0, \text{ then } \alpha[x,y]+\beta[z,w]+\gamma[y,z]+\delta[x,z]=[\alpha\delta^{-1}y+z,-\gamma y-\delta x+\beta w] \\
		&\text{ If } \delta=0, \gamma \neq 0 \text{ then }\alpha[x,y]+\beta[z,w]+\gamma[y,z]=[-\beta\gamma^{-1}w+y,\gamma z-\alpha x] 
	\end{aligned}\]
Now assume that $\delta=0=\gamma$. Set $N=\gen{u_2,\ldots,u_n}$. If $x \notin C_L(N)$, i.e.,  $[x,u_i]=t[z,w]$ for some $2 \leq i \leq n$ and $t \in k^*$, then $\alpha[x,y]+\beta[z,w]= [x,\alpha y+\beta t^{-1}u_i]$. If $x \in C_L(N)$ but $y \notin C_L(N)$ then, $[y,u_i]=t[z,w]$ for some $2 \leq i \leq n$ and $t \in k^*$, so $\alpha[x,y]+\beta[z,w]= [y,-\alpha x+\beta t^{-1}u_i]$. If $x,y \in C_L(N)$ but $w \notin C_L(N)$ then, $[w,u_i]=t[z,w]$ for some $2 \leq i \leq n$ and $t \in k^*$, so $\alpha[x,y]+\beta[z,w]= [x+w,\alpha y+\beta t^{-1}u_i]$. If $x,y,z,w \in C_L(N)$, then in this case $n \geq 3$ and $N'=I$ as $[N,L]=I$. Thus, $[u_i,u_j]=t[z,w]$ for some $2 \leq i,j \leq n$ and $t \in k^*$, so, $\alpha[x,y]+\beta[z,w]= [y+u_i,\alpha y+\beta t^{-1}u_j]$. Hence $w(L)=L'$.	 
\end{proof}

\section{$3$-step nilpotent Lie algebras}\label{sec5}

Observe that if $L$ is $3$-step nilpotent Lie algebra with $\dim L'=4$, then $L$ is minimally generated by atleast $3$ elements. Therefore, $\dim L \geq 7.$ We will prove that for $\dim L=7$, $w(L)=L'$ if $\dim Z(L) \leq 2$ and $w(L) \neq L'$ otherwise. From now onwards, $\delta \in \{0, 1\}.$

\begin{lem} \label{p7lem1}
	Let $L$ be a $7$-dimensional $3$-step nilpotent Lie algebra over field $k= \mathbb{F}_q$ of odd characteristic with $b(L)=3$, $\dim Z(L) \leq 2$  and $\dim L'=4.$ Then $w(L)=L'.$
\end{lem}
\begin{proof}By Lemma \ref{breadth3_classify} and \ref{ext_lemma}, $L$ is minimally generated by $3$ elements $x, y, z$ (say). Set $C = C_L(L').$ We will divide the proof into $4$ steps.
	
	\, \,
	   
\noindent{\bf Step 1.} {\it If $\dim Z(L) = 1$ then $C=L'.$}
	\begin{proof}
		As $0 \neq L^2 \subseteq Z(L)$ and $\dim Z(L) = 1$ so $Z(L)=L^2.$ Also $[L': L^2]=3$ therefore no non-zero element from the Lie algebra $\left\langle [x, y],  [x, z],  [y,z] \right\rangle$ can lie in $L^2.$ Since $L' \subseteq C$ implies that $\dim C \geq 4.$ If $\dim C=6$, then, without loss of generality, we can assume that $y,z \in C.$ Observe that $C_L(x)\cap L' \subseteq Z(L)$ and  $\dim (C_L(x)\cap L') =3$. So $\dim Z(L) \geq 3$.
		If $\dim C=5$, then, without loss of generality, we can assume that $z \in C.$ Observe that $C_L(x)\cap C_L(y)\cap L' \subseteq Z(L)$ and  $\dim (C_L(x)\cap C_L(y)\cap L') =2$. So $\dim Z(L) \geq 2$. Thus, $\dim C \neq 5$ or $6$. Hence, $C =L'.$
	\end{proof}
	\noindent{\bf Step 2.} {\it If $\dim Z(L) = 1$, then $w(L)=L'.$}
	\begin{proof}
		By Step 1 we have $C =L'.$ Observe that $\bar L := L/ L^2$ is $2$-step nilpotent Lie algebra of dimension $6$ on $3$- generators. Using the calculation of Theorem [\ref{dimL'3}, subcase(2a)] for $\bar L$ instead of $L$, we get, for $i \in k$, $$\bar L' =  \bigcup \limits_{ \delta , i} \ [ {\delta}\bar x+i \bar z, \bar L].$$
		We can interchange $x,y,z$ in the above equation because of symmetry. It is sufficient to show that $L^2 \subseteq \bigcap \limits_{ \delta , i} \ [ {\delta}u+i v, L]$ for some $u \neq v$ in $\{x, y, z\}$, where $i \in k$ such that $\delta$ and $i$ are not simultaneously zero. Firstly, suppose $[x,[x,y]] \neq 0$, i.e., $L^2=\gen{[x,[x,y]]}.$ 
		If $[z,[x,y]] \neq 0$, then $[z,[x,y]] = \mu[x,[x,y]]$ for some $\mu \in k^*.$ Taking $z'= z-\mu x$, gives a new generating set $\{x,y,z' \}$ for $L$ such that $[z',[x,y]]=0$ and $[x,[x,y]] \neq 0.$ So we can always assume that $[z,[x,y]]=0.$ Since $z \notin C$, either $[z,[y,z]]$ or $[z,[x,z]]$ is non-trivial. Therefore, for $\delta=0,1$ and $i \in k$, not simultaneously zero, we can easily see that $L^2 \subseteq \bigcap\limits_{ \delta , i} \ [{\delta}x+i z, L].$  
		
		Now, let us assume that $[x,[x,y]]=0.$ Then at least one of $[x,[y,z]]$ and $[x,[x,z]]$ is not trivial. If $[y,[x,y]]=0$, then $[z,[x,y]] \neq 0.$ Therefore, for $\delta=0,1$ and $i \in k$ not simultaneously zero, we get $L^2 \subseteq    \bigcap \limits_{ \delta , i} \ [{\delta}z+i x, L].$ If $[y,[x,y]] \neq  0$, then for $\delta=0,1$ and $i \in k$, not simultaneously zero, we can easily see that $L^2 \subseteq   \bigcap \limits_{ \delta , i} \ [{\delta}y+i x, L]$. 
	\end{proof}
	\noindent{\bf Step 3.} {\it If $\dim Z(L) = 2$ and $\dim L^2 = 1$, then $\dim C=5.$}
	\begin{proof}
		If $\dim C =6$, then as observed above $\dim Z(L) \geq 3$. If $\dim C = 4$, then $C = L'.$ We can assume that $[y,z] \in Z(L).$ If not, then $[y, z] = r [x, y]+s [x, z]$ modulo $Z(L)$ for some $r, s \in k.$ This implies $[y-s x, z+r x]=0$ under modulo $Z(L).$ Thus, $y'= y-s x$ and $z'= z+r x$ gives the new generating set $\{x, y', z'\}$ with the required property. Atleast one of $[x,[x, y]]$ and $[x, [x,z]]$ is non trivial; otherwise $x \in C$. We can assume that $[x,[x, y]] \neq 0$, i.e., $L^2=\gen{[x,[x,y]]}$. Also $[z,[x,y]]=[y,[x,z]]$ by Jacobi identity. If $[z,[x,y]]=0$ then $[z,[x,z]] \neq 0$ and $ [y,[x,y]] \neq 0$; otherwise $z,y \in C$. We can modifying the generating set such that $L'=\gen{[x,y],[x,z],[y,z],[x,[x,y]]}$ and $[y,z] \in Z(L), \, [z,[x,y]]=[y,[x,z]]=0=[x,[x,z]], \, [z,[x,z]] \neq 0, \, [y,[x,y]] \neq 0$. Indeed if $[x,[x,z]] \neq 0 $ then $[x,[x,z]]= \mu[z,[x,z]]$ for some $\mu \in k^*$, i.e., $[x-\mu z, [x,z]]=0$. Taking $x'=x-\mu z$ gives $\gen{x',y,z}$ as required generating set. Now $ [y,[x,y]] \neq 0$ means $ [y,[x,y]]= \lambda [x,[x,y]]$ for some $\lambda \in k^*$, i.e., $y-\lambda x \in C$, which is not possible.
		
		 So now assume that $[z,[x,y]] \neq 0$. We claim that then $[z,[x,z]]=0=[y,[x,y]]$. Suppose not, then $[z,[x,z]]=\mu_1 [y,[x,z]]$ for some $\mu_1 \in k^*$, i.e., $[z-\mu_1 y,[x,z]]=0$. Taking $z-\mu_1 y$ in place of $z$ gives generating set with $[z,[x,y]]=0$, which is not possible. If $[y,[x,y]] \neq 0$ then $[y-\mu_2 z, [x,y]]=0$ for some $\mu_2 \in k^*$. Again taking $y-\mu_2 z$ in place of $y$ gives generating set with $[z,[x,y]]=0$. So our claim is proved. We can modify the generating set such that $L'=\gen{[x,y],[x,z],[y,z],[x,[x,y]]}$ and $[y,z] \in Z(L), \, [z,[x,y]] \neq 0$, i.e., $[y,[x,z]] \neq 0$, $[x,[x,z]]=0$, $[y,[x,y]]=0$ and $[z,[x,z]]=0.$ Thus $[z,[x,y]]=\alpha [x,[x,y]]$ for $\alpha \in k^*$, i.e., $z-\alpha x \in C$, which is contradiction. So $\dim C \neq 4$. \end{proof}
	\noindent{\bf Step 4.} {\it If $\dim Z(L) = 2$ and $\dim L^2 = 1$, then $w(L) = L'.$}
	\begin{proof}
		By Step 3, we know that  $\dim C = 5.$ Assume that $z \in C.$ We discuss each case $[y,z] \in Z(L)$, $[x,z] \in Z(L)$ and $[y,z], [x,z] \notin Z(L)$ separately. Since $\bar L :=  L/L^2 $ is the  $2$-step nilpotent Lie algebra, then, as explained in  Step 2, to prove that $w(L) = L'$, it is sufficient to show that
		$$L^2  \subseteq \bigcap \limits_{ \delta , i} \ [ {\delta}u+i v,  L]$$
		for some $u \neq v$ in $\{x, y, z\}$, where $\delta=0,1$ and $i \in k$ such that $\delta$ and $i$ are not simultaneously zero.
		
		{\it Case (i).} Let $[y,z] \in Z(L)$ then $[x, [y,z]]=[y, [x, z]]=[z,[x, y]]=0.$ But $[x,[x, z]]\neq 0$ and $[y, [x, y]]\neq 0$. Thus for $\delta=0,1$ and $i \in k$, not simultaneously zero, we can see that
		$ L^2  = \bigcap \limits_{ \delta, i} \ [{\delta}x+i y, L'] \subseteq \bigcap \limits_{ \delta, i} \ [{\delta}x+i y, L].$ Hence $w(L)=L'$ by Lemma \ref{prelem3}.
		
		{\it Case (ii).} Next, assume that $[x,z] \in Z(L).$ Then $[x, [x, y]]$ and $[y, [y, z]]$ are non-trivial elements and $[x, [y, z]]=0.$ Thus for $\delta=0,1$ and $i \in k$, not simultaneously zero, we can easily see that
		$ L^2  = \bigcap \limits_{ \delta, i} \ [{\delta}y+ix, L'] \subseteq \bigcap \limits_{ \delta, i} \ [{\delta}y+ix, L].$ Hence $w(L)=L'$ by Lemma \ref{prelem3}.

		{\it Case (iii).} Let $[y, z],[x, z] \notin Z(L).$ Then $[y, z]$ and $[x, z]$ cannot be a multiple of each other. Indeed if $[y, z] = \mu[x, z]$ for some $\mu \in k^*$, then taking $\{x, y-\mu x, z\}$ as a generating set with $[y-\mu x, z] \in Z(L)$ arrives in Case (i). We can modify the generating set for $L$ to $\{x', y', z\}$ such that $[x', y'] \in Z(L)$ and $z \in C.$ If not, then modulo $Z(L)$, $[x, y] = \lambda_1 [x,z] + \lambda_2 [y, z]$ for some $\lambda_1, \lambda_2 \in k.$ So, $[ x+\lambda_2 z, y-\lambda_1 z]=0$ by taking modulo $Z(L)$. Taking $x'= x+\lambda_2 z, \, y' =y-\lambda_1 z$ gives us required generating set.
		
		Firstly, assume $[x',[x',z]]=0$ and $[y',[y',z]]=0.$ Then 
		$[x', [y',z]] = [y', [x',z]] \neq 0$ as otherwise $x',y' \in C$. Thus, for $\delta = 0, 1$ and $i \in k$, not simultaneously zero, we can easily check that
		$$ L^2  = \bigcap \limits_{ \delta, i} \ [{\delta}x'+i y', L'] \subseteq  \bigcap \limits_{ \delta, i} \ [{\delta}x'+iy', L].$$		
		
		Now, assume that $[y',[y',z]]$ or $[x',[x',z]]$ is non-trivial.We will only discuss $[y',[y', z]] \neq 0$ as other one follows similarly. Without loss of generality, we can assume $[x',[y',z]]=[y',[x',z]]= 0.$ Indeed $[x',[y',z]] = s [y',[y',z]]$, i.e., $[x'-s{y'},[y',z]] =0.$ Taking $\tilde{x}= x'-s{y'}$ gives generating set $\{\tilde{x}, y', z\}$ for which $[y', [\tilde{x},z]]=[\tilde{x},[y',z]]=0$, $[\tilde{x},y'] \in Z(L)$, $z \in C$ and $[y',[y',z]] \neq 0.$ For any $ t, i \in k$ we have $t[\tilde{x}, [ \tilde{x} ,z]] =[\tilde{x}+ib' , t[\tilde{x},z]]$ and $t[y',[y', z]]=[y' , t[y',z ]].$ Thus, for $\delta = 0, 1$ and $i \in k$, not simultaneously zero, we get
		$ L^2  = \bigcap \limits_{ \delta, i} \ [{\delta}\tilde x+i y', L'] \subseteq  \bigcap \limits_{ \delta, i} \ [{\delta}\tilde x+i y', L].$  Hence $w(L)=L'$ by Lemma \ref{prelem3}.
	\end{proof} Now, we will discuss the case where $\dim Z(L) = \dim L^2  = 2.$ In this case, we claim that $b(L) = 4$. Suppose not, i.e., $b(L) = 3.$ Then by Theorem \ref{breadth3_classify} there exists an ideal $I$ of $L$ such that $[L/I : Z(L/I)] = 3.$ Since $Z(L) = L^2 $ has dimension $2$, i.e., $L/I$ is $3$-step nilpotent. So, assume that $z+I \in Z(L/I).$ Observe that $[(L/I)':(L/I)^2]=2$, which cannot happen as $L/I$ possesses only two non-central generators. Hence, the claim as well as the lemma is proved.
	\end{proof}
The following lemma describes the case where $w(L) \neq L'$ for a $3$-step nilpotent Lie algebra of dimension $7.$
\begin{lem}\label{p7lem2} Here, $k=\mathbb{R}$ or $\mathbb F_{q}$, $\mathbb F_q$ is the field of odd characteristic. Let $L$ be a $7$-dimensional, $3$-step nilpotent Lie algebra of $b(L)=3$, $\dim Z(L) = 3$ and $\dim L'=4.$
	Then $w(L) \neq L'.$ Further, each element of $L'$ is sum of at most two elements of $w(L).$
\end{lem}
\begin{proof}
	Since $ \dim L/L'=3$, so let $L = \gen{x, y, z}.$ Then $L'= \langle [x,y], [y,z], [x,z],  L^2 \rangle.$ Also $\dim \frac{L'}{Z(L)}=1$, i.e., $L'/Z(L)= \gen{[x,y]}$ if $[x,y] \notin Z(L).$ Thus, $L^{2}=\gen{[x,[x,y]],[y,[x,y]}.$ So, $\dim L^2 \leq 2.$ We can take $Z(L)= \langle [y,z], [x,z], L^2) \rangle.$ Thus, by the Jacobi identity, $z \in C_L(L').$ Now, we consider two cases, namely, $\dim L^2$ is $1$ or $2.$

	If $\dim L^2 = 1$ then we can modifying the generating set $\{x, y, z\}$ such that $[x,[x,y]] \neq 0$ and $[y,[x,y]] = 0.$ Therefore,  $L^2 = \gen{[x,[x,y]]}.$ Observe that $[y,z]+[x,[x,y]] \notin w(L)$, by similar arguments as those used in [Lemma \ref{2-step_dim8}, part(1)].
	
	If $\dim L^2=2$ then $L^2= \langle [x,[x,y]],[y,[x,y]]\rangle.$ Observe that if neither $[x, z]$ nor $[y, z]$ lies in $L^2$ then one of them will be a scalar multiple of the other modulo $L^2.$ So assume, we have a generating set $\{x, y, z\}$ such that $[x, z] \in L^2.$ Let $[x,z] = \lambda_1[x,[x,y]] + \lambda_2[y,[x,y]] $ for some $\lambda_1,\lambda_2 \in k.$ Then $[x,z-\lambda_1[x,y]]=\lambda_2[y,[x,y]].$ Taking $z'=z-\lambda_1[x,y]$, we obtain a modified generating set $\{x, y, z'\}$ such that $[x,z']=\lambda_2[y,[x,y]]$ and $z' \in C_L(L').$ If $\lambda_2 =0$, then $[x,z'] = 0$, otherwise take ${\lambda_2}^{-1}z'$ in place of $z'$. Assume that $\lambda_2 = 1$, i.e., $[x,z']=[y,[x,y]].$
	
  Assume there exist $\alpha_i, \beta_i \in k$ for any given $\mu_1, \mu_2 \in k^*$ such that
	$${\mu_1}[y,z']+\mu_2[x,[x,y]] = [{\alpha_1}x+{\alpha_2}y+{\alpha_3}z'+{\alpha_4}[x,y], {\beta_1}x+{\beta_2}y+{\beta_3}z'+{\beta_4}[x,y]].$$
	After opening the Lie brackets and comparing the terms on both sides, we get
	\begin{eqnarray} \label{numeric6}
		\beta_2 \alpha_1 - \alpha_2 \beta_1& =& 0 \label{numeric6}\\
		\beta_3 \alpha_2 - \alpha_3 \beta_2 &=& \mu_1  \label{numeric7}\\
		\lambda_2(\beta_3 \alpha_1 - \alpha_3\beta_1)+\beta_4 \alpha_2 - \alpha_4 \beta_2&=&0 \label{numeric8}\\
		\beta_4 \alpha_1 - \alpha_4 \beta_1&=&\mu_2.\label{numeric9}
	\end{eqnarray}
	We have two cases, $\lambda_2=0$ and $\lambda_2=1.$ First, $\lambda_2=0$ is similar in calculation as in the previous paragraph. So, let us discuss case $\lambda_2=1.$ 
	If $\alpha_1=0$ then by (\ref{numeric9}), $\alpha_4, \beta_1$ both are non-zero. So (\ref{numeric6}), gives $\alpha_2=0.$ By (\ref{numeric7}),(\ref{numeric9}) we get $\beta_2=-\alpha_3^{-1}\mu_1, \beta_1 = - \alpha_4^{-1} \mu_2$, putting these in (\ref{numeric8}), we get ${(\alpha_3 \alpha_4^{-1})}^2=-\mu_1 \mu_2^{-1},$ but we can choose $\mu_1, \mu_2 \in k$ such that $(-\mu_1 \mu_2^{-1})$ is a non-quadratic residue. Thus, the claim is proved in this case. If $\alpha_1, \alpha_2$ are both non-zero then by doing similar calculation done in Lemma [\ref{2-step_dim8}, part(3)] will prove our claim.
\end{proof}

\begin{lem}\label{main1}
	Let $L$ be a $3$-step finite dimensional nilpotent Lie algebra such that $b(L)=3$, $\dim Z(L) = 3$, $\dim L'=4$ and  $\dim L \geq 8$. Then $w(L)=L'.$
\end{lem}
\begin{proof}
	As we observed in the first paragraph of the above Lemma \ref{p7lem2}, $\dim L^2 \le 2$.
	
	{\bf{Case(1)}} Let $\dim L^2  = 1$ then there exists an ideal $I$ of $L$ such that $[L/I : Z(L/I)] = 3$ by Theorem \ref{breadth3_classify}. If $I \neq L^2$, then by using \ref{keyresult} and Lemma \ref{ext_lemma}, we can find $2$-generator Lie subalgebra $M$ of $L$ such that $M'=L'$, which is not possible as $\dim (L'/L^2) = 3.$ So, $I = L^2$. By Theorem \ref{keyresult}, there exits a $3$-generator, $7$-dimensional Lie subalgebra $M$ of $L$ such that $M'=L'.$
	
	If $\dim L=8$, then $L=\gen{x, y, z, w}$ such that $M=\gen{x, y, z}$ and $[w, M] = I$ by Theorem \ref{keyresult}. If $\dim L \geq 9$ then, for some integer $n \ge 2$, $L = \gen{x, y, z, u_1, \ldots, u_n}$ such that $M = \gen{x, y, z}$ with $M' = L'$ and $[u_i, L] = I$ for $1 \le i \le n.$ We have the following two subcases:
	
	{\it{ Subcase(1a)}} $[u_i, M] = I$ for some $i \in \{1,\ldots,n\}$.
	
	{\it{ Subcase(1b)}} $[u_i, M] = 0$ for all $1 \le i \le n$.
	
	Let us discuss them one by one.
	
	{\it{ Subcase(1a)}} Let $[u_i, M] = I$ for some $i, \, 1 \le i \le n.$ Take $w:=u_i$ and $N=\gen{x,y,z,w}$. $N$ is $8$-dimensional Lie subalgebra of $L$ with $N'=L'$ and $[w,M]=I$. So, it is sufficient to study $N$, i.e., this case reduces to the former one where $\dim L=8$. If $0 \neq [x,[x,y]] \in L^{2}=I$ then take $I=\gen{[x,[x,y]]}$. We can modify the generating set for $M$ such that $L'= M'=\gen{ [x,y],[x,z],[y,z],[x,[x,y]]}$ with $[z,[x, y]] =0.$ Therefore $(L/I)'=\langle [\bar x, \bar y],[\bar x, \bar z],[\bar y, \bar z]\rangle$, where $\bar v = v+I$ for all $v \in L.$ Since $[w,M]=I$, therefore $[x,w]=\lambda_1[x,[x,y]]$ for some $\lambda_1 \in k$, then $[x,w-\lambda_1[x,y]]=0.$ Replacing $w$ by $w-\lambda_1[x,y]$, we can assume that $[x,w]=0.$ If $[z, w] \neq 0$, then $[z, w] =\lambda_2 [x, [x, y]]$ for some $\lambda_2 \in k^*.$ Let $\mu_i \in k$, $1 \leq i \leq 4$,
	\[
	\begin{aligned}
		&\text{ If } \mu_1 \neq 0 \text{ then } {\mu_1}[\bar x,\bar y]+{\mu_2}[\bar y,\bar z]+{\mu_3}[\bar x,\bar z]=[\bar x-\mu_2{\mu_1}^{-1} \bar z, {\mu_1}\bar y+{\mu_3}\bar z]\\
		&\text{ If } \mu_1= 0 \text{ then }{\mu_2}[\bar y,\bar z]+{\mu_3}[\bar x,\bar z]=[\bar z, -\mu_3 \bar x-\mu_2 \bar y]\\
	\end{aligned}
	\]
	Further, \begin{eqnarray*}
	 {\mu_4} [x,[x,y]] &=& [x- \mu_2 {\mu_1}^{-1} z, {\mu_4}[x,y]].\\
	 {\mu_4} [x,[x,y]] &=& [z, {\mu_4}\lambda_2^{-1}w].
	\end{eqnarray*}
 We can easily check that for $\delta = 0, 1$ and $i \in k$, 
	$$L'/I = \bigcup \limits_{\delta, i} \ [{\delta}{\bar x}+ i \bar z, \bar L] \text{ and } I \subseteq \bigcap \limits_{\delta, i}  \ [{\delta} x+i z, L],$$
	where $i$ and $\delta$ are not simultaneously zero.
	Hence, $w(L) = L'$ by Lemma \ref{prelem3}. Lastly, if $[z, w] =0$ then $[y , w] \neq 0.$ By similar arguments, the result holds for this case also. 
			
	{\it{ Subcase(1b)}} Let $[u_i, M] = 0$ for all $1 \le i \le n$ and $\dim L \geq 9$. Then $L$ is a central product of $M$ and $K := \gen{u_1, \ldots, u_n}.$ $K$ is non-abelian as $Z(L) \subseteq L'$. So $\exists \, u_i, u_j \in K$ for some $1 \leq i<j \leq n$ such that $I = \gen{[u_i, u_j]}$. Then $N:=\gen{x, y, z, w, v}$, where $w = u_i$ and $v = u_j$, is $9$-dimensional subalgebra with $N'=L'.$ For $i \in k$, we can easily show that,
	$$N'/I =  [\bar z+ \bar w, \bar N]\bigcup \limits_{i} \ [ \bar x+i \bar z+ \bar w, \bar N] \  $$
	
	$$I \subseteq \ [z +w, N] \bigcap \limits_{ i} \ [ x +i z+ w, N]  ,$$ 
	where  $\bar a = a+ I$ for $a \in N.$ Hence, $w(N) = N'$ by Lemma \ref{prelem3}.

	{\bf{Case(2)}}{\it{ Let $\dim L^2 = 2.$}} We can show that $I \nsubseteq L^2$ as in the above case. Then, by Theorem \ref{keyresult}, either (i) there exists a $2$-generator, $5$-dimensional, $3$-step nilpotent Lie subalgebra $M$ such that $L$ is central product of $M$ and $K$ with $\dim K' = 1$ or (ii) there exists a $3$-generator, $7$-dimensional, $3$-step nilpotent Lie subalgebra $M$ such that $M' = L'$ and $L=M+K$ where $K$ is at most $2$-step nilpotent subalgebra of $L$. In case(i), $w(L) = L'$ by Lemma \ref{prelem4}. So assume (ii).

	If $\dim L = 8$, then $L = \gen{x, y, z, w}$ such that $M = \gen{x, y, z}$ and $\gen{[w, M]} = I.$ If $\dim L \geq 9$ then, for some integer $n \ge 2$, $L = \gen{x, y, z, u_1, \ldots, u_n}$ such that $M = \gen{x, y, z}$ and $[u_i, L]= I$ for $1 \le i \le n.$ 
	First, assume that $[u_i, M] = I$ for some $1 \le i \le n.$ Then the subalgebra $N:=\gen{x, y, z, w}$, where $w = u_i$, is $8$-dimensional such that $[w, M] = I$ and $N' = L'.$ As observed above, it is sufficient to study $N$, i.e., the case where $\dim L= 8.$ 
	
	Using the arguments of Lemma \ref{p7lem2}, assume that $I= \gen{[y, z]}$,
	$$M^2 = \gen{[x, [x, y]], [y, [x,y]]},$$
	$[x, z] \in M^2$ and $z \in C_M(M').$  If $[y,w]=\lambda_1[y,z]$ for some $\lambda_1 \in k$, then $[y,w-\lambda_1 z]=0.$ Set $w'= w-\lambda_1 z.$ If $[z,w']=\lambda_2[y,z]$ for some $\lambda_2 \in k$, then $[z, w'+\lambda_2 y] =0.$ Changing $w$ by $w'+\lambda_2 y$, we can assume that $[y, w] = [z, w] = 0.$  Let $[x,w] =\lambda_3[y,z]$ for some $\lambda_3 \in k^*.$ Using arguments from the preceding subcase(1a), it is not difficult to get:
	$$ L'/I = \bigcup \limits_{ \delta, i} \ [{\delta} \bar x+i {\bar y}, \bar L] \text{ and } I \subseteq  \bigcap\limits_{ \delta, i} \ [{\delta} x+i y, L] ,$$
	where $\bar a = a +I$ for $a \in L$ and $i \in k$ and $\delta = 0, 1$ such that $i$ and $\delta$ are not simultaneously zero. Hence, $w(L) = L'$ by Lemma \ref{prelem3}.
	
	Now, consider the case $[u_i, M] = 0 \, \forall \, 1 \le i \le n.$ Here $L$ is a central product of $K := \gen{u_1, \ldots, u_n}$ and $M$. $K$ is non-abelian as $Z(L) \subseteq L'$. So $\exists \, u_i, u_j \in K$ for some $1 \leq i<j \leq n$ such that $I = \gen{[u_i, u_j]}$. Then $N:=\gen{x, y, z, w, v}$, where $w = u_i$ and $v = u_j$, is ${9}$-dimensional subalgebra such that $N'=L'.$ For $i \in k$, we can easily see that 
	$$N'/I = \ [\bar y+\bar w,\bar N] \bigcup \limits_{i} \ [\bar x +i \bar y+ \bar w ,\bar N],  $$
	$$I  \subseteq   [y +w , N]\bigcap \limits_{ i} \ [x+i y+ w , N].$$ Thus, $w(N) = N'$.
\end{proof}

\section{$4$-step nilpotent Lie algebras}\label{sec6}
 From \cite{graaf2007}, $L_{6,21}(0), L_{6,21}(\epsilon)$, where $\epsilon \neq 0$ are only $4$-step nilpotent Lie algebras of breadth $3$ with $4$-dimensional derived subalgebra up to isoclinism. We will discuss these in the following lemma.

\begin{lem}\label{p6lem}
Let $L$ be a $6$-dimensional, $4$-step nilpotent Lie algebra over field $k=\mathbb{F}_q$, field of odd Characteristic with $\dim L'=4.$ Then $w(L) = L'$ if and only if $\dim Z(L) = 1.$ Furthermore, if $w(L) \neq L'$, then each element of $L'$ is sum of at most two elements of $w(L).$ More precisely, If $L=L_{6,21}(\epsilon)$, where $\epsilon \neq 0$ then $w(L) = L'.$ If $L=L_{6,21}(0)$ then $w(L) \neq L'$ and each element of $L'$ is sum of at most two elements of $w(L).$
\end{lem}

\begin{proof} Since $L_{6,21}(0)$ and $ L_{6,21}(\epsilon)$, where $\epsilon \neq 0$ are only $4$-step nilpotent Lie algebras (upto isoclinism) of dimension $6$ whose derived subalgebra has dimension $4.$ If $L=L_{6,21}(0)$ then its presentation is 
\begin{equation*}
L=	\left\lbrace u_1,u_2,u_3,u_4,u_5,u_6: \begin{aligned} &[u_1,u_2]=u_3,[u_1,u_3]=u_4,\\
&[u_1,u_4]=u_6, [u_2,u_3]=u_5
\end{aligned} \right\rbrace
\end{equation*}
Observe that $Z(L)=\left\langle u_5,u_6 \right\rangle$ .i.e., $\dim Z(L)=2$ and $L$ is minimally generated by $u_1,u_2.$ So, $[u_1,u_4]+[u_2,u_3] \notin w(L)$ by the same arguments used in Lemma [\ref{2-step_dim8}, part(1)]. Thus, $w(L) \neq L'.$ Further, by applying Theorem \ref{dimL'3} and Lemma \ref{prelem5} on $L/I$ where $I=<u_5>$, we can write every element of $L'$ as a sum of at most two elements of $w(L)$. 
If $L=L_{6,21}(\epsilon)$, where $\epsilon \neq 0$, then it has presentation 
\begin{equation*}
	L_{6,21}(\epsilon) =\left\lbrace u_1,u_2,u_3,u_4,u_5,u_6:\begin{aligned} &[u_1,u_2]=u_3,[u_1,u_3]=u_4,[u_1,u_4]=u_6,\\ &[u_2,u_3]=u_5,[u_2,u_5]=\epsilon u_6
	\end{aligned} \right\rbrace.
\end{equation*}
Here $Z(L)=\left\langle u_6 \right\rangle$, i.e., $\dim Z(L)=1$,  $L$ is minimally generated by $u_1,u_2$ and $L' =<u_3,u_4,u_5,u_6>$ with $L^3=<u_6>.$
\[
\begin{aligned}
	& {\begin{aligned} \text{If } \beta \neq 0 \text{ then } \alpha u_3+ \beta u_4+\gamma u_5 
	 &=\alpha[u_1,u_2]+\beta [u_1,u_3]+\gamma [u_2,u_3]\\
	 &=[u_1+\gamma \beta^{-1}u_2, \alpha u_2+\beta u_3].
	\end{aligned}}\\
	&\text{If } \beta=0 \text{ then } \alpha u_3 +\gamma  u_5 =\alpha[u_1,u_2]+\gamma [u_2,u_3]= [u_2,-\beta u_1+\gamma u_3].
\end{aligned}
\]
\begin{eqnarray*}
	\text{Hence } L'/L^3 &= & [u_2, \bar{L}] \bigcup\limits_{\substack{{\beta \neq 0}\\{\gamma \in k}}} [\bar{u_1} +\gamma \beta^{-1} \bar{u_2}, \bar{L}] \\
	 \text{and } L^3 &\subseteq & [u_2, L] \bigcap\limits_{\substack{{\beta \neq 0}\\{\gamma \in k}}} [u_1 +\gamma  \beta^{-1} u_2, L]
\end{eqnarray*}Hence $w(L)=L'$ by Lemma \ref{prelem3}.
\end{proof}\medskip

\begin{lem} \label{p7lem3}
Let $L$ be a $7$-dimensional $4$-step nilpotent Lie algebra over $k=\mathbb{F}_q$, field of odd characteristic, with $b(L)=3$ and $\dim L'=4.$ Then $w(L)=L'.$
\end{lem}

\begin{proof}
Since $Z(L) \subseteq L'$, so $\dim Z(L) \leq 4.$  \begin{enumerate}
\item[•] If $\dim Z(L)= 4$ then $L$ is a $2$-step nilpotent Lie algebra, which is not the case.
\item[•] If $\dim Z(L)= 3$ then $L$ is a $3$-step nilpotent Lie algebra, which is not the case.
\end{enumerate}
So $\dim Z(L) \leq 2.$ By Theorem \ref{keyresult}, there exists a $4$-step, $6$-dimensional nilpotent Lie subalgebra $M$ of $L$ such that $M' = L'.$ If $\dim Z(M)= 1$, then $w(M) = M' = L'$  by Lemma \ref{p6lem}. Thus, $w(L) = L'.$ If $\dim Z(M) = 2$ then there exists an ideal $I$ of $L$ such that $|L/I:Z(L/I)|=3$ by Theorem \ref{breadth3_classify}.Take $L = \left\langle x, y, z \right\rangle$ such that $M = \left\langle x, y \right\rangle.$ Therefore, $I = L^3$ using the arguments given in the proof of Theorem \ref{keyresult}(last paragraph). Thus, $\bar{L} := L/I = \left\langle \bar{x}, \bar{y}, \bar{z} \right\rangle $ is $3$-step nilpotent such that $\bar{z} \in Z(\bar{L})$. So, $[z, L] = I$ as $z \notin Z(L)$. Observe that $\bar{L}' = \left\langle [\bar x, \bar y], [\bar x, [\bar x, \bar y]],  [\bar y, [\bar x, \bar y]]\right\rangle .$ By doing similar calculation as in Theorem [\ref{dimL'3}, subcase(2c)], we can get, $\bar{L}' =   \bigcup \limits_{{\delta},{i}} \ [{\delta}{\bar{x}}+i{\bar{y}} , \bar L]$ where for $i \in k$ and $\delta = 0, 1$ such that $i$ and $\delta$ are not simultaneously zero.

Observe that $M$ is in the isoclinism class $L_{6,21}(0)$ of \cite{graaf2007}. So, we can assume that $I=\gen{[y, M^2]}$. Since $C_M(z)$ is maximal, we can modify $z$ such that $[y,z]=0$ and $I=\langle [x,z] \rangle.$ Then $I \subseteq \bigcap\limits_{\substack{{\delta \in \{0,1\}}\\{i \in k}}} [{\delta}x+iy , L]$, where $i$ and $\delta$ cannot be simultaneously zero. Hence, $w(L)=L'$ by Lemma \ref{prelem3}.
\end{proof}

 \begin{lem}\label{4-step nilpotent} Let $L$ be a $4$-step nilpotent Lie algebra of dimension at least $7$ over field $k=\mathbb{F}_q$ of odd characteristic with $b(L)=3$, $\dim L'=4.$ Then $w(L)=L'.$
 \end{lem}
 \begin{proof}
Given $\dim L \geq 7$ and $\dim L'=4$, $Z(L) \subseteq L'$ then $\dim Z(L)\leq 2$ as $\dim Z(L)=3$ and $[L': Z(L)]=1.$ By Lemma \ref{maximal_nilpotency_class}, $L$ is at most $3$-step nilpotent, which contradicts the given hyothesis. By Theorem \ref{keyresult}, there exists a $2$-generator Lie subalgebra $M$ of $L$ such that $\dim M = 6$, $M' =L'$ and $L=M+K$ with $\dim K \leq 1.$ Also $K = \gen{u_1, \ldots, u_n}$, for some $n \ge 1$ such that $[u_i, L] = I$ for $1 \le i \le n.$ If $[u_i, M] = I$ for some $1 \le i \le n$, then the subalgebra $L_1:=\gen{x, y, z}$, where $z = u_i$ is $7$-dimensional such that $L_1' = L'.$ Hence,$w(L) = L'$ by Lemma \ref{p7lem3}. Now, assume that $[u_i, M] = 0$ for all $1 \le i \le n.$ Also $K$ is non-abelian as $Z(L) \subseteq L'$. Thus, $L$ is a central product of $M$ and $K.$

{\textbf{Case(1)} } If $\dim Z(L) = 1$ then $\dim Z(M)= 1$. So, $w(M) = M'$ by Lemma \ref{p6lem}. Hence, $w(L) =L'.$

{\textbf{Case(2)} } If $\dim Z(L) = 2$ then $K' = I$ and so $\exists \, u_i, u_j \in K$ for some $1 \leq i<j\leq n$, such that $I = \gen{[u_i, u_j]}$. Let $N:= \gen{x, y, z, w}$, where $z= u_i$ and $w = u_j.$ Then $L' =N'.$ Consider $\bar N := N/I.$ As we saw
$(\bar{N})' = \gen{[\bar x, \bar y], [\bar x, [\bar x, \bar y]], [\bar y, [\bar x, \bar y]]}$ in the proof of Lemma \ref{p7lem3}. 
We can easily show that $i, j \in k$, $(\bar{N})' = \bigcup \limits_{{i},{j}} \ [{i\bar{x}}+j {\bar{y}}+ \bar z, \bar N]$ and $I \subseteq \bigcap\limits_{{i},{j}} [i x+j y+ z, N].$ Hence, $w(N)=N'$ by Lemma \ref{prelem3}, and therefore we have $w(L)=L'.$
 \end{proof}
 

\section{Proof of Theorem \ref{thmA}}\label{sec7}

\noindent {\it Proof of Theorem \ref{thmA}.} Let $L$ be a finite-dimensional nilpotent Lie algebra over $k$ such that $\dim L'=4.$ Also, let $Z(L) \subseteq L'$. Observe that $L$ is at most $5$-step nilpotent. By Remark \ref{remark}, we have  $b(L) \ge 3.$ Observe that $\dim L \geq 6.$ If $b(L) = 4$, then $w(L) = L'.$ Therefore, we assume that $b(L) = 3.$ When $L$ is $2$-step nilpotent, the assertion follows from Lemmas \ref{2-step_dim8} and \ref{cl2lem2}. Now, let $L$ be a $3$-step nilpotent. There is no $3$-step nilpotent Lie algebra of dimension $6$ satisfying the given hypothesis. Therefore, $\dim L \geq 7.$ If $\dim Z(L) \leq 2$, then $w(L) = L'$ according to the Lemma \ref{p7lem2}. If $\dim L \geq 8$ and $\dim Z(L) = 3$, then by Lemma \ref{main1}, we have $w(L) = L'.$ Finally, if $\dim L = 7$ and $\dim Z(L) = 3$, then by Lemma \ref{p7lem2} we have $w(L) \neq L'.$ It only remains to deal the cases where $L$ is $4$-step or $5$-step nilpotent.

Let $L$ be a $4$-step nilpotent. If $\dim L = 6$, then it follows from Lemma \ref{p6lem} that $w(L) = L'$ if and only if $\dim Z(L)= 1.$ So assume that $\dim L \geq 7.$ Then, using Lemma \ref{4-step nilpotent}, $w(L)=L'.$

Finally, assume that $L$ is $5$-step nilpotent. Our claim is $b(L) = 4$. Suppose, $b(L)\neq 4$, i.e., we can assume that $b(L) = 3.$ Since $L/Z(L)$ is $4$-step nilpotent and $Z(L) \subseteq L'$, we have $\dim Z(L) =1.$ Hence, by Theorem \ref{breadth3_classify}, $I=Z(L)$ is the only choice such that $[L/I:Z(L/I)] =3 $. Taking  $Z_2(L):=\{x \in L: [x,y] \in Z(L) \, \forall \, y \in L\}$, we get $[\bar L : \bar{Z_2(L)}] = 3$, where $\bar L= L/Z(L), \bar{Z_2(L)}= Z_2(L)/ Z(L)$, i.e., $\dim L/Z_2(L) = 3$, which is not true as $L/Z_2(L)$ is a $3$-step nilpotent. Hence, our claim is proved, and the proof of the theorem is complete. \hfill $\Box$

\bibliographystyle{plain}
\bibliography{arxiv_file}
\end{proof}
\end{document}